\documentclass[12pt]{amsart}
\usepackage{amssymb}
\usepackage{import}
\usepackage[english]{babel}
\usepackage{tikz}
\newtheorem{theorem}{Theorem}[section]

\newtheorem{lemma}{Lemma}[section]

\usepackage{thmtools}
\usepackage{thm-restate}

\numberwithin{mytheorem}{section} 

\numberwithin{mylemma}{section} 

\numberwithin{mycorollary}{section} 
\usepackage{hyperref}
\DeclareMathOperator{\pred}{pred}
\DeclareMathOperator{\spl}{sp}
\DeclareMathOperator{\suc}{succ}
\usepackage{cleveref}

\usepackage{color}

\newtheorem{remark}[theorem]{Remark}

\newcommand{\DD}{\mathcal{D}}

\newcommand{\cc}{\mathfrak{c}}

\author{ Jorge Bruno and Paul J. Szeptycki}

\newcommand{\N}{\mathbb{N}}
\title[Proof of TAC for TMR]{A Proof of The Tree Alternative Conjecture Under the Topological Minor Relation}

\begin{document} 
\maketitle

\begin{abstract} In 2006 Bonato and Tardif posed the Tree Alternative Conjecture (TAC): the equivalence class of a tree under the embeddability relation is, up to isomorphism, either trivial or infinite. In 2022 Abdi, et al. provided a rigorous exposition of a counter-example to TAC developed by Tetano in his 2008 PhD thesis. In this paper we provide a positive answer to TAC for a weaker type of graph relation: the topological minor relation. More precisely, letting $[T]$ denote the equivalence class of $T$ under the topological minor relation we show that
\smallskip
\begin{enumerate}
\item $|[T]| = 1$ or $|[T]|\geq \aleph_0$ and
\smallskip
\item  $\forall r\in V(T)$, $|[(T,r)]| = 1$ or $|[(T,r)]|\geq \aleph_0$.
\end{enumerate}
\smallskip
In particular, by means of {\it curtailing} trees, we show that for any tree $T$ with at least one ray with infinitely many vertices with degree at least $3$: $|[T]| \geq 2^{\aleph_0}$.
\end{abstract}
\section{Introduction} 

 In \cite{A} Bonato and Tardif proved that $|[T]| = 1$ or $\infty$ for any rayless tree and conjectured the same must be true of any tree: the Tree Alternative Conjecture (TAC) stated that the number of isomorphism classes of trees mutually embeddable with a given tree $T$ is either 1 or infinite. Since mutually embeddable finite trees are necessarily isomorphic, TAC is becomes true for finite trees, and it has been confirmed for a number of nontrivial classes of infinite trees -  \cite{A}, \cite{LPS} and \cite{T}. In particular, in \cite{T} it is shown that TAC holds for all rooted trees, where a {\it rooted tree} $(T,r)$ is one with a distinguished vertex $r\in V(T)$. In 2022 a counterexample to TAC was found by Abdi et al. \cite{counter} where for each $n\in \N$ an unrooted and locally finite tree $T_n$ is constructed with $|[T_n]| = n$.
 
 
 Herewith we present a proof of TAC with respect to the {\it topological minor relation} $\leq^{\sharp} $, where for trees $T,S$ we have $T \leq^{\sharp} S$ if some subdivision of the tree $T$ is isomorphic to a subgraph of $S$. The rooted topological minor notion is defined analogously to the unrooted case. In this paper we prove the following theorem, where $[T]$ represents the equivalence class of $T$ under the topological minor relation.
  
  \noindent
  {\bf Theorem 1.} For any tree $T$, 
\begin{enumerate}
\item $|[T]| = 1$ or $|T|\geq \aleph_0$ and
\smallskip
\item  $\forall r\in V(T)$, $|[(T,r)]| = 1$ or $|[(T,r)]|\geq \aleph_0$.
\end{enumerate}

\medskip
The techniques developed and employed to prove Theorem 1 heavily rely on the fact the trees are, as discussed in the Background, {\it well-quasi-ordered} (wqo) under the topological minor relation. Although Tyomkyn in \cite{T} establishes TAC for all rooted trees under the embeddability relation, there appears to be a strong connection between wqo and TAC for unrooted trees and other similar structures.

The layout of the paper is quite simple: in Section~\ref{sec:TACforLarge} we prove Theorem 1 for {\it large trees} by employing the method of \textit{curtailing} - developed by the present authors; and Section~\ref{sec:TACforSmall} deals with all {\it small trees} by adapting the techniques developed by Bonato and Tardif in \cite{A}, Halin in \cite{H}, and Polat and Sabidussi in \cite{P} to the topological minor relation.

\section{Background}

\subsection{Graph-theoretic background} Most notation is standard and can be found in \cite{D} but we present a brief summary of key concepts and non-standard terminology employed in this paper. A {\it rooted tree} $(T,r)$ is composed of a tree $T$ and a distinguished vertex $r\in V(T)$. The isomorphism relation on rooted and unrooted trees is denoted by $\simeq$ where, in the rooted case, an isomorphism must send root to root. The {\it splitting number} of a vertex $v$, $\spl(v)$, in a rooted tree $(T,r)$ is simply $\deg(v) -1$ if $v\not = r$ and $\deg(v)$ otherwise.  Any rooted tree $(T,r)$ generates a partial ordering $\leq_T$ on its set of vertices by establishing that $s\leq_T t$ provided that the unique path from $r$ to $t$ contains $s$; if $s <_T t$ then we say that $s$ is {\it below} $t$ and that $t$ is {\it above} $s$. This is called the {\it tree order} on $(T,r)$ and defines a meet semilattice order on all vertices of $(T,r)$; $(T,r)$ is closed under all non-empty meets. We denote the meet of a collection $C \subseteq V(T)$ by $\wedge_T C$ and $v_1 \wedge_T \ldots \wedge_T v_n$ if $C$ is finite.  A tree $(T,r)$ is a {\it rooted subtree} of a tree $(S,s)$ if $T$ is a subtree of $S$ and $v \leq_S w \Longleftrightarrow v\leq_T w$ for all $v,w\in V(T)$. Given a vertex $v \in V(T)$ the {\it full subtree of $(T, r)$ rooted at $v$}, $(T_v, v)$, is the rooted subtree of $(T, r)$ induced by all vertices $w$ with $v \leq_T w$. A {\it subdivision} of a tree $T$ (resp. of a rooted tree $(T,r)$) is any tree (resp. rooted tree) obtained by subdividing any number of its edges. A {\it path}, $p:v_1,\ldots,v_n$, is a sequence of distinct pairwise adjacent vertices (i.e., $v_i$ adjacent to $v_{i+1}$) and we demand $v_i <_T v_{i+1}$ for all $i\leq n-1$ when $p$ is in a rooted tree $(T,r)$. The length of a path $p$, $|p|$, is simply the number of vertices in it. A path $p:v_1,\ldots,v_n$ is {\it bare} if $\deg(v_i)\leq 2$ for all $i\leq n$, and {\it maximal bare} if $p$ is bare and not contained within any other bare path. For vertices $v \leq_T w$ in a rooted tree $(T,r)$, $d_T(v,w) = |p| -1$ where $p:v,\ldots,w$ is the unique path joining said vertices. For the case where $v=r$ we denote $d_T(v,w) = l_T(w)$ and refer to it as the {\it level} of $w$. For a vertex $v$ we use $\pred(v)$ to denote the unique vertex immediately preceding $v$ and $\suc(v)$ to denote the set of immediate successors of $v$. If $\spl(v)=1$ then we use $\suc(v)$ to denote the unique vertex succeeding $v$ instead of the set containing it. Given a ray $R = v_1v_2\ldots$ we use the notation $R_n = v_n$ and $R_n^\uparrow = v_nv_{n+1}\ldots$. A ray $R$ is {\it eventually bare} if for some $N\in \N$, $\deg(R_n) = 2$ for all $n\geq N$, and {\it bare} if $N=1$. A rooted tree is said to be {\it large} if it contains at least one ray that is not eventually bare. A rooted tree is then {\it small} if it's not large. We also define an unrooted tree $T$ to be {\it large} (resp. {\it small}) if $(T,r)$ is large (resp. small) for any $r\in V(T)$. 

\subsection{Topological minor relation and wqos}Given trees $T$ and $S$, an injective map  $\phi: V(T)\to V(S)$ is a {\it minor embedding} if $\phi$ can be extended to an isomorphism between a subdivision of $T$ and the smallest subtree $S'$ of $S$ containing all vertices in $\phi(V(T))$. If there exists such a minor embedding $\phi$ then we say that $T$ is a {\it topological minor} of $S$ and we write $T \leq^\sharp S$. We use the shorter term {\it embedding} when referring to a minor embedding provided there is no danger of ambiguity. If $T \leq^\sharp S$ and $T \geq^\sharp S$ then we write $T \equiv^\sharp S$ and say that they are {\it topologically equivalent} or of the same {\it topological type}. The equivalence class of topological types of a tree is denoted by $[T]$ and its size by $|[T]|$. These topological notions are defined in a similar way for rooted trees. An injective map  $\phi: V(T)\to V(S)$ between rooted trees $(T,r)$ and $(S,s)$ is a {\it rooted minor embedding} if $\phi$ can be extended to an isomorphism between a subdivision of $(T,r)$ and the smallest rooted subtree $(S',s')$ of $(S,s)$ containing all vertices in $\phi(V(T))$. A tree $(T,r)$ is {\it self-similar} provided that there exists $v >_T r$ with $(T,r) \leq^\sharp (T_v,v)$. An equivalent definition for a minor embedding $\phi: V(T)\to V(S)$ is that of a {\it meet semilattice homomorphism}:  for any pair $w,v \in V(T)$, $\phi(v \wedge_T w) = \phi(v) \wedge_S \phi(w)$. If there exists an embedding $\phi: V(T)\to V(S)$ then we say that $(T,r)$ is a {\it rooted topological minor} of $(S,s)$ and we write $(T,r) \leq^\sharp(S,s)$. If $(T,r) \leq^\sharp
(S,s)$ and $(T,r) \geq^\sharp (S,s)$ then we write $(T,r) \equiv^\sharp (S,s)$ and say that they are {\it topologically equivalent}. The equivalence class of topological types of a tree is denoted by $[(T,r)]$ and its cardinality by $|[(T,r)]|$.

A quasi-ordered set $(X,\leq)$ is {\it well-quasi-ordered} (wqo) if it is well-founded and all antichains are finite. The rooted topological minor relation is a wqo on the collection of trees (\cite{Kurs2} and \cite{NW}). We make use of the following equivalent characterisations for a quasi-ordered set $(X,\leq)$ to be a wqo \cite{Mil}: 
\begin{itemize}
\item any sequence $x_1, x_2, \ldots$ in $X$  contains a pair $x_i \leq x_j$ with $i<j$; and
\item any sequence $x_1, x_2, \ldots$ in $X$ has an increasing subsequence $x_{n_1} \leq x_{n_2}\leq \ldots$.
\end{itemize}
\begin{remark} Since the rooted topological minor relation well-quasi orders rooted trees, for trees distinct from the ray or double ray, the notion of being large is equivalent to having a self-similar subtree. In particular, not having a self-similar subtree is equivalent to being small.
 \end{remark}
 \subsection{Set-theoretic background} We use some basic set-theoretic notation and terminology which can be found in \cite{JW} or \cite{Ku}. Recall that $\aleph_0$ is the cardinality of the natural numbers and we use the set-theoretic notation $\omega$ to denote the the set of natural numbers, $\N$, with $0$ included. The {\it continuum} $\cc$ or $2^{\aleph_0}$ denotes the cardinality of the set of real numbers (which is also equal to the cardinality of $\N^\N$, the set of all functions  $f:\N\rightarrow \N$). An {\it almost disjoint family} (a.d.f.) $\DD$ is a collection of infinite subsets of $\N$ so that for any pair $X,Y \in \DD$, $|X\cap Y| \in \N$. It is well know that there exists an almost disjoint family of size $\cc$ (pg. 159, \cite{Ku}).

\section{Large Trees}\label{sec:TACforLarge}

In this section all trees are assumed to be large. We begin by introducing the notion of curtailing and subsequently employ it to prove Theorem 1 for all large trees. The outline is as follows: for a fixed tree $T$ and $r\in V(T)$ we create - by means of curtailing $(T,r)$ - a tree $(T',r)$ of the same topological type as $(T,r)$ and $2^{\aleph_0}$ subdivisions of $(T',r)$, $(T'_\kappa,r)$ with $\kappa \in 2^{\aleph_0}$, of the same topological type as $(T',r)$ but $T'_\kappa\not \simeq T_\gamma'$ for any pair $\kappa\not = \gamma$. Since $T'_\kappa\equiv^\sharp T$ for all $\kappa$ we obtain the following result.

\begin{restatable}{thm}{TAClargeUnrooted}\label{thm:TAClargeUnrooted} For any large tree $T$:
\begin{enumerate}
\item $|[T]| \geq 2^{\aleph_0}$ and
\item $\forall r\in V(T)$, $|[(T,r)]|\geq 2^{\aleph_0}$.
\end{enumerate}
\end{restatable}

\subsection{Curtailing}
\setcounter{theorem}{3}
\setcounter{lemma}{1}
\setcounter{corollary}{4}

Fix a rooted tree $(T,r)$ and a self-similar tree $(S,s)$ distinct from the ray. For any $u \in V(T)$ with $(S,s) \equiv^\sharp (T_u,u)$ and $\spl(u) = 1 < \spl(\pred(u))$ let $p_u: u,\ldots, u_n$ be the maximal bare path starting from $u$; such a path must always exist since $(S,s)$ is not small. Now define $(T',r)$ as the tree obtained from $(T,r)$ by simultaneously replacing any such path $p_u$ with an edge $e_u = \{\pred(u), \suc(u_n)\}$. That is, simultaneously suppress all vertices in all paths $p_u$. We say that $(T,r)$ is \textit{curtailed into $(T',r)$ via $(S,s)$}. By design, all vertices of splitting number greater than $1$ in $(T,r)$ are also present in $(T',r)$. It is entirely possible that $(T,r) \equiv^\sharp (S,s)$ and that $\spl(r)=1$. In which case, the above construction would delete $r$. By Lemma~\ref{lem:pathTopType} below the tree obtained from deleting the maximal bare path starting at $r$ yields a topologically equivalent tree to $(T,r)$. Hence, without loss of generality, we assume that $\spl(r)>1$. Of course, if $(S,s) \not \equiv^\sharp (T_u,u)$ for any $u\in V(T)$ then curtailing $(T,r)$ via $(S,s)$ results in $(T,r)$ itself and, trivially, Lemma~\ref{lem:sameTopType} and Theorem~\ref{thm:curtailDoesntChange} are true. We begin by proving Lemma~\ref{lem:pathTopType} .

\begin{figure}
  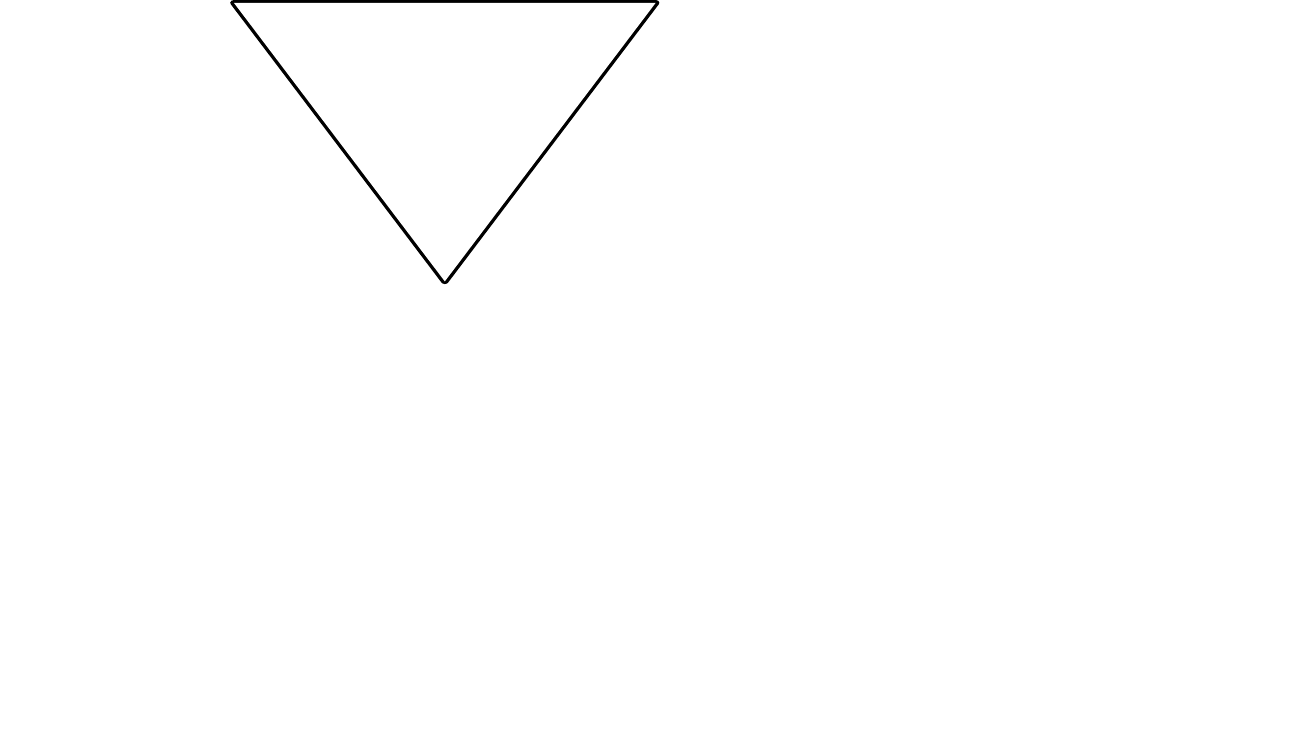
 \caption{Turning a maximal bare path $p_u$ into an edge $e_u$.}
 \label{fig:test}
\end{figure}

\begin{lemma}\label{lem:pathTopType} Let $(T,r)$ be self-similar. It follows that the tree $(T',r_1)$ that results from adding a bare path $p:r_1,\ldots,r_n$ below $r$ is of the same topological type as $(T,r)$.
\end{lemma}

\begin{proof} We need only show that  $(T',r_1) \leq^\sharp (T,r)$. Since $(T,r)$ is self-similar we can create a sequence $r<_T \phi(r) <_T \phi^2(r) <_T \ldots$ where $\phi: V(T) \to V(T)$ witnesses $(T,r)$'s is self-similarity. Choose any $k$ so that given the path $p':w_1 = r,\ldots, w_j = \phi^k(r)$ we have $j>n = |p:r_1,\ldots,r_n|$. Let $\psi:V(T') \to V(T)$ with $\psi(v) = \phi(v)$ for all $v\geq_{T'} r$ and $\psi(r_i) = w_j$. Then $\psi$ is clearly a minor embedding.
\end{proof}

\begin{lemma}\label{lem:sameTopType} Let $(T',r)$ be $(T,r)$ curtailed via $(S,s)$. For any $v\in V(T)$ with $\spl(v) >1$ if $(T_v,v) \equiv^\sharp (S,s)$, then $(T'_v,v) \equiv^\sharp (S,s)$.
\end{lemma}

\begin{proof} Since $(T_v,v)$ is a subdivision of $(T'_v,v)$ then $(T'_v,v)\leq^\sharp (S,s)$ whenever $(T_v,v) \leq^\sharp (S,s)$. For the reverse inequality, let $v\in V(T)$ with $\spl(v) >1$, $(T_v,v) \leq^\sharp (S,s)$ witnessed by a $\phi: V(S) \to V(T_v)$ and, by Lemma~\ref{lem:pathTopType}, assume that $\spl(s) >1$. Next, we construct a rooted minor embedding $\psi: V(S) \to V(T'_v)$ witnessing $(T'_v,v) \geq^\sharp (S,s)$. In order to so do, fix a minor rooted embedding $\eta: V(S) \to V(S)$ witnessing $(S,s)$'s self-similarity (i.e., $\eta(s) >_{S} s)$), set 
\[
L_n = \{w\in V(S) \mid l_{S}(w) = n\}
\]
and put $(S_n,s)$ as the full subtree of $(S,s)$ generated from all vertices in $L_0 \cup \ldots \cup L_n$. We construct $\psi$ by inducting on the $(S_n,s)$'s. More precisely, for any $k\in \N$ we focus on the least $n_k\in \N$ for which $\psi(u)$ is not defined for some $u\in V(S_{n_k})$. By the minimality of $n_k$, $\psi\upharpoonright (S_{n_k},s)$ will be a rooted topological minor embedding.\\

\noindent
\textbf{Base Case:} for any $u\in L_1$ if $\phi(u)\in V(T'_v)$, then we let $\psi(u) = \phi(u)$; since $\spl(s) >1$ then $\phi(s) \in V(T'_v)$ and $\psi(s) = \phi(s)$. Otherwise, $\spl(\phi(u)) =1$ and $(S,s) \equiv^\sharp (T_{\phi(u)},\phi(u))$. That is, $\phi(u)$ was removed when constructing $(T',r)$ because $(S,s) \equiv^\sharp (T_{\phi(u)},\phi(u))$. Let $\phi_u: V(S) \to V(T_{\phi(u)})$ witness $(S,s) \leq^\sharp (T_{\phi(u)},\phi(u))$ and consider  the increasing sequence $\phi_u\circ \eta(s) <_{T}\phi_u\circ \eta^2(s)<_{T}\ldots$. If $(S_u, u)$ is a bare ray $u=u_1u_2\ldots$ then extend $\psi$ by letting $\psi(u_i) = \phi_u\circ \eta^i(s)$. Otherwise, let $p_u: u = u_1, \ldots, u_m$ be the maximal bare path in $(S,s)$ starting with $u$ and extend $\psi$ by letting $\psi(u_i) = \phi_u\circ \eta^i(s)$ and $\psi(u_{m+1}) = \phi_u\circ \eta^m(u_{m+1})$. It is simple to verify that $\psi$ is indeed a meet semilattice homomorphism on $(S_1,s)$\\

\noindent
\textbf{Inductive Case:} at stage $k\in \N$, let $n_k\in \N$ be the smallest so that there exists a vertex $u \in V(S_{n_k})$ where $\psi$ is undefined and so that $\psi: (S_i,s) \to (T,r)$ defines a meet semilattice homomorphism for all $i< n_k$. By construction, either 
\begin{itemize}
\item $\psi(\pred(u)) = \phi_w\circ\eta^j(\pred(u))$, for some $w<_S u$ and $j\in \N$, or
\item $\psi(\pred(u))  = \phi(\pred(u))$. 
\end{itemize}
Set $u^*$ to be $\phi(u)$ or $\phi_w\circ\eta^j(u)$ depending on the case above that is true. If $u^* \in V(T')$ then set $\psi(u) = u^*$.  Otherwise, as with the base case, it must be that $\spl(u^*) =1$ and $(S,s) \leq^\sharp (T_{u^*},u^*)$ with embedding $\phi_{u}: V(S)\to V(T_{u^*})$. In which case, consider  the increasing sequence $\phi_u\circ \eta(s) <_{T}\phi_u\circ \eta^2(s)<_{T}\ldots$. If $(S_u, u)$ is a bare ray $u=u_1u_2\ldots$ then extend $\psi$ by letting $\psi(u_i) = \phi_u\circ \eta^i(s)$. Otherwise, let $p_u: u = u_1, \ldots, u_m$ be the maximal bare path in $(S,s)$ starting with $u$ and extend $\psi$ by letting $\psi(u_i) = \phi_u\circ \eta^i(s)$ and $\psi(u_{m+1}) = \phi_u\circ \eta^m(u_{m+1})$. 

To show that $\psi$ is a meet semilattice homomorphism on $(S_n,s)$ we must verify a few cases separately. Firstly, observe that by construction for any pair $w,w'\in V(S_{n_k})$ it follows that $w\leq_S w'$ if, and only if, $\psi(w)\leq_{T'}\psi(w')$. If $\pred(w) \not = \pred(w')$, then 
\[
\psi(w\wedge_S w') = \psi(\pred(w) \wedge_S \pred(w')) = \psi(\pred(w))\wedge_{T'}\psi(\pred(w'))
\]
 by the inductive hypothesis and 
 \[
 \psi(\pred(w))\wedge_{T'}\psi(\pred(w')) = \psi(w) \wedge_{T'} \psi(w')
 \]
  since $\psi(w)\bot_{T'} \psi(w')$, $\psi(\pred(w)) \leq_{T' }\psi(w)$ and $\psi(\pred(w')) \leq_{T'} \psi(w')$. Otherwise we have that $\pred(w) = \pred(w')$. If $\psi(\pred(w)) = \phi(\pred(w))$, then, by construction, it follows that $\psi(w) \geq_T \phi(w)$ and $\psi(w') \geq_T \phi(w')$. Hence, 
  \[ 
  \psi(w) \wedge_{T'} \psi(w') = \phi(w) \wedge_T \phi(w') = \phi(w \wedge_S w') = \psi(w \wedge_S w').
  \]
   Finally, if $\psi(\pred(w)) = \phi_t\circ\eta^j(\pred(w))$ for some $t\leq_S \pred(w)$ and $j \in \N$, then $\psi(w) \geq_T  \phi_t\circ\eta^j(w)$ and $\psi(w') \geq_T  \phi_t\circ\eta^j(w')$. Therefore, $\psi(w) \wedge_{T'} \psi(w') = \psi(w) \wedge_{T} \psi(w') = \phi_t\circ\eta^j(w) \wedge_T \phi_t\circ\eta^j(w') = \phi_t\circ\eta^j(w \wedge_S w') = \psi(w \wedge_S w').$
\end{proof}

\begin{theorem}\label{thm:curtailDoesntChange} Let $(T',r)$ be $(T,r)$ curtailed via $(S,s)$. It follows that for any $v\in V(T')$, $(T'_v,v) \equiv^\sharp (T_v,v)$. In particular, $(T,r) \equiv^\sharp (T',r)$.
\end{theorem}
\begin{proof} Clearly, $(T'_v,v) \leq^\sharp (T_v,v)$. For the reverse inequality, notice that if $(S,s) \equiv^\sharp (T_v,v)$, then by Lemma~\ref{lem:sameTopType} we are done. Otherwise, much in the same spirit as with Lemma~\ref{lem:sameTopType}, we construct a minor embedding $\psi:V(T_v) \to V(T_v')$ by inducting on the full subtrees $(T_v^n,v)$ of $(T_v,v)$ generated by $L_0 \cup \ldots \cup L_n$ where
\[
L_n = \{w\in V(T_v) \mid l_{T_v}(w) = n\}.
\]

\noindent
\textbf{Base Case:} since $L_0 = \{v\}$ and, by assumption, $(S,s) \not \equiv^\sharp (T_v,v)$ then $v\in V(T'_v)$ and we let $\psi(v) = v$. Take any $u\in L_1$. If $u\in V(T'_v)$ then we let $\psi(u) = u$. Otherwise, if $u\not \in V(T')$ then $(S,s) \equiv^\sharp (T_u,u)$ and $\spl(u) =1$. Let $p_u: u=u_1,\ldots, u_{n}$ be the maximal bare path starting at $u$ and put $\suc(u_n) = u^*$. Since $(S,s) \equiv^\sharp (T_u,u)$ then Lemma~\ref{lem:pathTopType} yields $(T_u,u) \equiv^\sharp (T_{u^*},u^*)$ and by Lemma~\ref{lem:sameTopType}, $(S,s) \equiv^\sharp (T'_{u^*},u^*)$. In turn, $(T_u,u) \equiv^\sharp (T'_{u^*},u^*)$ and we extend $\psi$ to all $w\geq_T u$ by mapping $w\mapsto \phi_u(w)$ where $\phi_u: V(T_u) \to V(T'_{u^*})$ witnesses $(T_u,u) \leq^\sharp (T'_{u^*},u^*)$. It is clear that $\psi$ restricted to $(T_v^1,v)$ defines a meet semilattice homomorphism.\\

\noindent
\textbf{Inductive Case:} set $k \in \N$ as the smallest for which we can find a vertex $u \in V(T_v^{n_k})$ where $\psi$ is not defined. This implies that $\psi(\pred(u)) = \pred(u)$ and we apply the same logic as with the base case; distinguish between whether or not $(S,s) \equiv^\sharp (T_u,u)$ and extend $\psi$ accordingly. The meet semilattice homomorphism property of $\psi$ is verified by cases as with Lemma~\ref{lem:sameTopType}.
\end{proof}

\subsection{Proof of Theorem~\ref{thm:TAClargeUnrooted}} Let us begin by restating the theorem.

\TAClargeUnrooted*

\begin{proof}

(1) Start with a large tree $T$, a not eventually bare ray $v_0v_1\ldots$ witnessing this. Root $T$ at $r = v_0$ and put $v_k$ as the first vertex along the ray with $(T_{v_k},v_k)\leq^\sharp (T_{v_j},v_j)$ for some $j>k>0$ as rooted subtrees of $(T,r)$. Set $s=v_k$ and $(S,s) = (T_{v_k},v_k)$ and by Lemma~\ref{lem:pathTopType} we can assume that both  $\spl(s) > 1$ and $\spl(r) >1$. Since $(S,s)$ is self-similar there exists an embedding $\phi:V(S) \to V(S)$ witnessing this. Set $r_1 = s$, $r_n = \phi^{n}(s)$ for $n>2$, and put $D: r_1 \ldots $ as the unique ray in $(T,r)$ containing all $r_n$'s. Observe that $\phi$ establishes that $(T_{r_i},r_i) \equiv^\sharp (T_{r_j},r_j)$ for all $i,j \geq 1$ and that $\phi^i(V(D_i^\uparrow))\subseteq V(D_i^\uparrow)$. Let $(T',r)$ denote the curtailed tree $(T,r)$ via $(S,s)$ and denote the ray $D'$ in $(T',r)$ as the obtained from $D$ in $(T,r)$. A moments thought reveals that one can modify $\phi$ into an embedding $\eta:V(T'_{r_1}) \to V(T'_{r_1})$ with $\eta(r_i) = r_{i+1}$ and $\eta(V(D^{'\uparrow}))\subseteq V(D^{'\uparrow})$. This is since all vertices of degree at most $3$ in $D$ remain in $D'$ (e.g., all roots $r_i$ remain in $D'$) and by Theorem~\ref{thm:curtailDoesntChange} all subtrees of $(T_{r_1},r_1)$ rooted on vertices off the ray $D$ remain topologically equivalent to their curtailed versions. Set $e_n = \{\pred(r_n),r_n\}$ in $(T',r)$ for each $n\in \N$, and for each $f \in \N^\N$ we let $(T'_f,r)$ denote the tree generated from $(T',r)$ by replacing the edge $e_n$ with a bare path of length $f(n)$. Let $D'_f$ denote the subdivision of $D'$ in $(T'_f,r)$. Observe that by curtailing $(T,r)$ and subsequently subdividing $D'$ the only bare paths in $D'_f$ above $\pred(r_1)$ are those bare paths ending in an $r_n$. Next we prove that
\begin{itemize}
\item $(T'_f,r) \equiv^\sharp (T,r)$ for all $f\in \N^\N$ (hence $T'_f \equiv^\sharp T$), and
\item $\exists \mathcal{D} \subset \N^\N$ of size $\mathfrak{c}$ where $T'_f \not \simeq T'_g$ for all $f,g\in \mathcal{D}$.
\end{itemize}


\begin{lemma}\label{lem:segmentToRay} Assume that $T'_f \simeq T'_g$ for some $g,f \in \N^\N$ with bijection $\psi^f_g : V(T'_f) \to V(T'_g)$ witnessing this. It follows that $\psi^f_g$ must map a final segment of $D'_f$ entirely within $D'_g$. 
\end{lemma}
\begin{proof} 
Assume that $\psi^f_g$ does not map a final segment of $D'_f= u_1\ldots$ entirely within $D'_g$, put 
\[
m = \min\{n\in \N \mid  \forall j\geq n , \psi^f_g(u_j) \not \in D_g' \text{ and } \psi^f_g(u_j) \leq_{T'_g} \psi^f_g(u_{j+1})\}
\]
 and $B= b_1\ldots$ the ray in $(T'_g,r)$ containing all vertices $\psi^f_g(u_n)$ with $n\geq m$ (i.e., the ray $B$ is disjoint from $D_g'$, as in Figure 2). Recall that all of the $r_n$ (i.e., those used to define $D$ in $(T,r))$ are in $D'_f$. Let $r_l \in V(D'_f)$ so that $\phi^f_g(r_l) = b_j$ for some $j>2$ (we can choose an arbitrarily large value of $j$ since the $\psi^f_g(r_i)$ are cofinal in $B$ for $i\geq l$). It follows that the rooted tree $([T'_g]_{b_j},b_j)$ (i.e., the full subtree of $(T'_g,r)$ rooted at $b_j$) is isomorphic to $(T_{r_l},r_l)$. Indeed, $m$ was chosen so that $r_l >_{T_g'} r_{l-1}$ and, thus, $\psi^f_g$ preserves the order of $(T_{r_l},r_l)$ in  $(T'_g,r)$. In turn,  $([T'_g]_{b_j},b_j)$ and $r_l >_{T_g'} r_{l-1}$ are of the same topological type as $(S,s)$. But since $\spl(\pred(b_j))>1$ in $(T'_g,r)$ and $\spl(\pred(r_l)) = 1$ by design of $(T'_f,r)$ this contradicts that $\psi^f_g$ is an isomorphism.
\end{proof}
\begin{figure}
  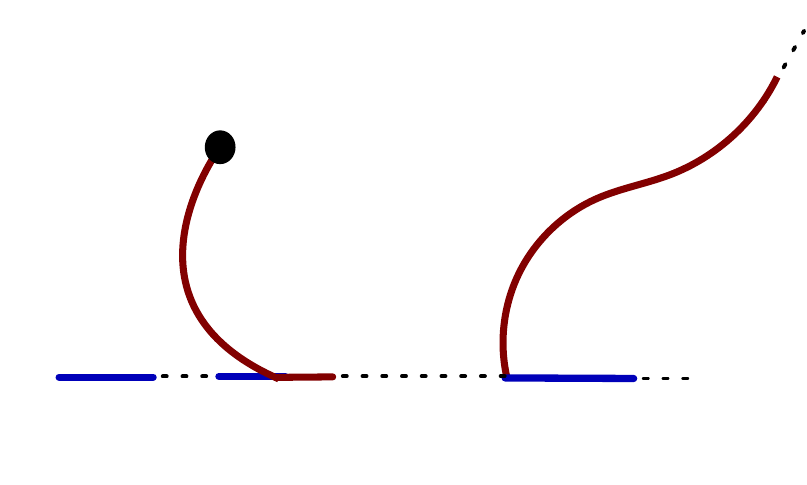
 \caption{The ray $B$ in green with initial vertex $b_1$. Also the initial segment of the ray $\psi^f_g(D'_f)$ and the ray $D'_g$ highlighted in red and blue, resp.}
 \label{fig:test}
\end{figure}
Next, we focus on proving that $T'_f \not \simeq T'_g$ for continuum many $f,g\in \N^\N$. By Lemma~\ref{lem:segmentToRay} given $f,g\in\N^\N$ with $\psi^f_g : V(T'_f) \to V(T'_g)$ witnessing $T'_f \simeq T'_g$ it follows that $\psi^f_g$ must map a final segment of $D'_f$ entirely within $D'_g$. Given that $\psi^f_g$ maps bare paths to bare paths, the final segment of $D'_f$ that is mapped entirely within $D_g'$ must have its bare paths of the same length as those in $D'_g$.  Let $\DD$ be an a.d.f. of size $\cc$ and for each $X=\{x_1,x_2,\ldots\} \in \DD$ set $f_X(n) = x_n$. Then $T'_{f_X} \simeq T'_{f_Y}$ precisely when $X=Y$. To finish the proof of Theorem~\ref{thm:TAClargeUnrooted} (1) we establish the following result.

\begin{lemma} For all $f \in \N^\N$, $T_f' \equiv^\sharp T$.
\end{lemma}

\begin{proof} Notice that curtailing $(T'_f,r)$ via $(S,s)$ yields $(T',r)$ and, thus, by Theorem~\ref{thm:curtailDoesntChange} the result follows.
\end{proof}

(2) The construction of each $(T'_f,r)$, for $f\in\N^\N$, based on an unrooted tree $T$ can be applied directly to $(T,r)$ by selecting any ray $D$ in $(T,r)$ that is not eventually bare. This means that there are continuum many unrooted trees $T_f' $ with $T_f' \not \simeq T$ (hence, $(T_f',r) \not \simeq (T,r)$) but with $(T_f',r) \equiv^\sharp (T,r)$.

\end{proof}

\section{Small Trees}\label{sec:TACforSmall}

We now turn our attention to small trees. We first prove Theorem 1 for all rooted small trees and extend the result to all small trees by adapting a fixed-point result from Polat and Sabidussi to the topological minor relation.

\subsection{Rooted small trees}
Unlike the scenario encountered for large trees, given a pair small trees $(T,r) \equiv^\sharp (S,s)$ (distinct from the ray) any rooted minor embedding $\psi$ witnessing $(T,r) \leq^\sharp (S,s)$ must map $r$ to $s$. Indeed, otherwise for any $\phi$ that witnesses $(T,r) \geq^\sharp (S,s)$ it follows that the sequence $(\phi\circ\psi)^n(r)$ spans a ray in $T$ that is not eventually bare. This observations then allows us to directly employ the arguments developed in \cite{T} and \cite{A} to the rooted topological minor relation. In the following lemmas for a rooted tree $(T,r)$, $S(T) = \{(T_u,u) \mid u \in \suc(r)\}$.

\begin{lemma}\label{lem:greaterThanSubtree} For any $(T_u,u) \in S(T)$, $|[(T_u,u)]| \leq |[(T,r)]|$. 
\end{lemma}
\begin{proof} The proof of this lemma is exactly that of Lemma 2 in \cite{T} but we reproduce it here. Fix a $v\in \suc(r)$ and take any $(U,u) \in [(T_v,v)]$ (i.e., a non-isomorphic tree of the same topological type as $(T_v,v)$). Construct the tree $(T_U,r)$ as follows: for any $w\in \suc(v)$
\begin{itemize}
\item replace $(T_w,w)$ with $(U,u)$ if $(T_w,w) \equiv^\sharp (U,u)$, and
\item leave $(T_w,w)$ as it is, otherwise.
\end{itemize}
By design, $(T_U,r) \equiv^\sharp (T,r)$ and $(T_U,r) \not \simeq (T,r)$. Moreover, for any other $(W,w) \in [(T_v,v)]$ non-isomorphic to $(U,u)$ it follows that $(T_U,r) \not \simeq (T_W,r)$.
\end{proof}
In view of the above and given a small tree $(T,r)$ next we prove that if $|[(T_v, v)]| =1$ for all $v \in \suc(r)$, then $ |[(T,r)]| = 1$ or $\geq \aleph_0$. This would then complete the proof of Theorem 1 for small rooted trees. Indeed, by Lemma~\ref{lem:all1}, that $\aleph_0 > (T,r) >1$ implies the same of at least one $(T_v, v)$ for $v\in \suc(r)$, which in turn implies the same of some $(T_w, w)$ for $w\in \suc(v)$, and so on. This would create a sequence $\left(|[(T_{v_n}, {v_n})]|\right)_{n\in\N}$ with $v_n \in \suc(v_{n-1})$ and $|[(T_{v_n}, {v_n})]|>1$. A contradiction since any such sequence in a small tree ends on a terminal vertex or spans an eventually bare ray (i.e., a tail of all $1$'s).

\begin{lemma}\label{lem:all1} Let $(T,r)$ be small. If for all $(T_u,u) \in S(T)$, $|[(T_u,u)]| = 1$ then $ |[(T,r)]| =1$ or $\geq \aleph_0$. 
\end{lemma}
\begin{proof} Lemma 2 in \cite{A} after exchanging rooted subgraph with rooted topological minor but we also reproduce a simplified version here. Let $(T,r)$ and $(U,u)$ be two non-isomorphic trees with $(U,u)\in [(T,r)]$, and rooted embeddings $\phi_T: V(T) \to V(U)$ and $\phi_U: V(U) \to V(T)$ witnessing this. Further assume that $|[(T_v, v)]| =1$ for any $v \in \suc(r)$. We seek to prove that $|[(T,r)]| \geq \aleph_0$. 

Put $\mathcal{I}= \{(X_\alpha,\alpha) \mid \alpha\in \kappa\}$  as the set of all isomorphism types of all rooted trees $(T_v,v)$ and $(U_w,w)$ with $v\in \suc(r)$  and $w\in \suc(u)$. Define $f_T: \{(T_v,v)\mid v\in \suc(r)\} \to \kappa$ so that $f_T(T_v,v) = \alpha$ provided $(T_v,v) \simeq (X_\alpha, \alpha)$ (resp. $f_U: \{(U_w,w)\mid w \in \suc(u)\} \to \kappa$ so that $f_U(U_w, w) = \alpha$ provided $(U_w, w) \simeq (Y_\alpha, \alpha)$). Since $(T,r) \not \simeq (U,u)$ there must exist a $\beta \in \kappa$ for which, without loss of generality, $ |f_T^{-1}(\beta)| < |f_U^{-1}(\beta)|$. Put $\eta = |f_T^{-1}(\beta)|$ and $\zeta = |f_U^{-1}(\beta)|$, and enumerate the elements of $f_T^{-1}(\beta)$ and $f_U^{-1}(\beta)$ as $(T_i,i)$ and $(U_j,j)$ with $i\in \eta$ and $j\in \zeta$. Before moving further with this proof, let us establish some more notation: put $(T^*,r)$ as $(T,r)$ but with all full subtrees $(T_i,i)$, $i \in \eta$, removed - of course, this also includes the roots $i$ of each $(T_i,i)$ and the edge joining $i$ with $r$. In what follows we fix a $(T_k,k) \in f_T^{-1}(\beta)$ and for each cardinal $\lambda$ denote $(T^*_\lambda,r)$ as $(T^*,r)$ with $\lambda$ copies  ($(T^*_{\lambda_l}, {\lambda_l})$ with $l\in \lambda$)  of $(T_k,k)$ attached to $r$. Clearly, $(T^*_\eta,r) \simeq (T,r)$ and $(T^*_\nu,r) \not \simeq (T^*_\xi,r)$ for any pair of cardinals $\nu \not = \xi$.

Fix a $\gamma \in \zeta$ for which $\phi_U(U_\gamma,\gamma) \subseteq (T_v,v)\not \in f_T^{-1}(\beta)$. Next, for each $n\in \omega$, put $w_n \in \suc(s)$ so that $(\phi_T\circ\phi_U)^n[(U_\gamma,\gamma)] \subseteq (U_{w_n}, w_n)$ and $v_n\in \suc(r)$ with $\phi_U\circ(\phi_T\circ\phi_U)^n[(U_\gamma,\gamma)] \subseteq (T_{v_n}, v_n)$. Clearly, $(U_{w_0},w_0) = (U_\gamma,\gamma)$ and $(T_{v_0}, v_0) = (T_v, v)$, and 
\[
(U_{w_0},w_0) \leq^* (T_{v_0},v_0) \leq^*(U_{w_1},w_1) \leq^*(T_{v_1},v_1) \leq^* \ldots
\]
 Observe that $w_n \not = w_0$ for any $n\in \omega$ since then $(U_\gamma, \gamma) \simeq (T_v,v)$ from the assumption that $|[(T_u, u)]| =1$ for any $u \in \suc(r)$. The same logic dictates $(U_{w_n},w_n) \not \in f_U^{-1}(\beta)$ and $(T_{v_n},v_n) \not \in f_T^{-1}(\beta)$ for any $n>1$. Thus, we get two sequences of non-repeating elements: $(w_n)_{n\in\omega}$ and $(v_n)_{n\in\omega}$. \\

If $\eta$ is finite, then for each $\lambda \geq \eta$ it follows that $(T^*_{\lambda},\lambda) \geq^\sharp (T,r)$. For any for each $\lambda \in (\eta,\infty)$ put $\phi_\lambda: V(T^*_{\lambda}) \to V(T)$ with
\begin{itemize}
\item $\phi_\lambda((T^*_{\lambda_l},{\lambda_l})) = \phi_U\circ(\phi_T\circ\phi_U)^l[(U_\gamma,\gamma)] $ for each $l \leq \lambda$, 
\item $\phi_\lambda((T_{v_n},{v_n})) = \phi_U\circ(\phi_T\circ\phi_U)^{n+\lambda+1}[(U_\gamma,\gamma)]$ for each $n \in \omega$, and
\item the identity everywhere else.
\end{itemize}
Since $(U_\gamma,\gamma) \equiv^\sharp (T_k,k)$, $\phi_\lambda$ is a witness of $(T^*_{\lambda},\lambda) \leq^\sharp (T,r)$ for each finite $\lambda \geq \eta$ but $(T^*_{\lambda},\lambda) \not \simeq (T,r)$ for any $\lambda > \eta$. If $\eta \geq \aleph_0$, then since $\zeta > \eta$ one can easily adapt the proof of finite $\eta$ to create $\zeta$ many non-isomorphic $(T^*_{\lambda},\lambda) $, $\lambda \in \zeta$, of the same topological type as $(T,r)$.

\end{proof}

\subsection{Unrooted small trees}

Given a rayless tree $T$ the authors of \cite{P} showed that there exist a vertex or edge of $T$ fixed by all automorphisms of $T$. By modifying the arguments used in the above to minor embeddings one can easily obtain an analogous result for all small trees different to the ray or double ray. Firstly, let us remark that a minor self-embedding of a finite tree $T$ is necessarily an automorphism and, as such, it must either fix a vertex or an edge of $T$ - \cite{H}, Lemma 2. In fact, more is true: given a set $S \subseteq V(T)$ let $\overline{S}$ denote the subtree of $T$ generated by the union of all paths joining any pair of vertices in $S$ , $F_T = \{v\in V(T_\lambda)\mid \deg(v)>2\}$ and $\text{TME}(T)$ as the collection of all topological minor embeddings of $T$. Let $T$ be locally finite, small, infinite and distinct from the ray or double ray. It follows that $F_T\not = \emptyset$ is finite (since $T$ is small) and any minor embedding $\psi$ induces a permutation of $F_T$ where $\deg(v) = \deg(\psi(v))$ for all $v\in F_T$. It then follows that $\psi\upharpoonright\overline{F_T}\in \text{TME}(\overline{F_T})$ and, thus, that $\psi$ induces an automorphism of $\overline{F_T}$. Since $\overline{F_T}$ is finite then $\psi$ restricted to $\overline{F_T}$ is actually an automorphism and $\overline{F_T}$ has either a fixed edge or vertex by all $\psi \in \text{TME}(T)$. It then follows that $T$ itself has a fixed edge or vertex for all $\psi \in \text{TME}(T)$.

\begin{lemma}\label{lem:fixLocFiniteandSmall} Any small and locally finite tree $T$, distinct from the ray or double ray, has an edge or vertex fixed by all $\psi \in \text{TME}(T)$.
\end{lemma}

We extend the above to all small trees by means of employing the following lemma from \cite{P}.

\begin{lemma}\label{lem:Polat}[Lemma 1.1, \cite{P}] Let $G$ be a rayless graph, $(A_\alpha)_{\alpha\in \text{Ord}}$ a decreasing sequence of subsets of $V(G)$ so that 
\begin{enumerate}
\item $A_\alpha = \bigcap_{\beta < \alpha} A_\beta$ for any limit $\alpha$ and
\item each $A_\alpha$ induces a connected subgraph of $G$.
\end{enumerate}
By $\kappa$ denote the smallest ordinal for which the sequence $(A_\alpha)_{\alpha\in \text{Ord}}$ becomes constant. If $A_\kappa =\emptyset$, then $\kappa$ is a successor ordinal.
\end{lemma}

\begin{lemma}\label{lem:fixed} Given a small tree $T$, distinct from the ray or double ray, there exists a vertex or edge of $T$ fixed by all $\psi \in \textrm{TME}(T)$.
\end{lemma}
\begin{proof} Given Lemma~\ref{lem:fixLocFiniteandSmall}, let $T$ be small but not locally finite, and set $T_1 = \overline{\text{infinite}(T)}$, where $\text{infinite}(T) \subset V(T)$ is the set consisting all vertices of infinite degree. Clearly, $T_1$ is connected, rayless, properly contained in $T$ - since $T_1$ does not contain any end-points or rays from $T$ - and for any $\psi\in \text{TME}(T)$ we have $\psi\upharpoonright T_1 \in \text{TME}(T_1)$. For all $\alpha \in \text{Ord}$ construct $T_\alpha$ as

\[
T_\alpha =
\begin{cases} \overline{\text{inf}(T_\beta)} & \text{ if $\alpha = \beta + 1$ and}\\

 \bigcap_{\beta < \alpha} T_\beta,& \text{ otherwise}
\end{cases}
\]

\medskip

Let $\kappa$ denote the smallest ordinal for which $(T_\alpha)$ is constant. This happens when $T_\kappa = \emptyset$. Indeed, otherwise the sequence $(T_\alpha)$ would not be constant. Therefore, by Lemma~\ref{lem:Polat}, that $\kappa = \lambda +1$ for some $\lambda \in \text{Ord}$. In turn, $T_\lambda$ is locally finite and small and any such $\psi$ defines an automorphism on $T_\lambda$. Hence, there must exist a vertex of edge of $T$ fixed by all $\psi \in \text{TME}(T)$.
\end{proof}

Notice that Lemma~\ref{lem:fixed} establishes a dichotomy in the following sense; if an edge $e= \{r,s\}$ in a tree $T$ is fixed by all minor embeddings with $\psi(r) = s$, for all $\psi\in \text{TME}(T)$, then $T$ does not have a fixed vertex - the converse is also true. Indeed, all minor embeddings swap $s$ with $r$ precisely when any vertex in $V$ is sent to the opposite connected component of $T\smallsetminus\{e\}$. This observation is important for the proof of the following theorem.

 \begin{theorem} For any small tree $T$, $|[T]| = 1$ or $|[T]| \geq \aleph_0$. 
 \end{theorem}
\begin{proof} The proof of Theorem 1 in \cite{A} applies directly to this theorem and, for completeness, we reproduce it below.

 Let $T$ be small and assume that $|[T]| > 1$. Next we show that $|[T]|\geq \aleph_0$. Since $T$ is small then, by Lemma~\ref{lem:fixed}, $T$ has either a fixed vertex $r$ or edge $\{r,s\}$. Let $T^*$ be a non-isomorphic topologically equivalent tree to $T$, and $\psi_T$, $\psi_{T^*}$ denote minor embedding witnessing $T \leq^\sharp T^*$ and $T \geq^\sharp T^*$, respectively. Consider the tree $(T,r)$ and observe that $(T,r)$ and $(T^*,\psi_T(r))$ are not isomorphic. If $r$ is a fixed vertex of $T$ then $(T,r)$ and $(T^*,\psi_T(r))$ are of the same topological type as witnessed by $\psi_T$, $\psi_{T^*}$. Hence, $|[(T,r)]| \geq \aleph_0$. If $\{r,s\}$ is a fixed edge then the rooted embeddings $\psi_T$ and $\psi_{T^*}\circ \psi_{T}\circ \psi_{T^*}$ witness $(T,r)\equiv^\sharp(T^*,\psi_T(r))$ and $|[(T,r)]| \geq \aleph_0$ as well. In either case, $(T,r)$ is topologically equivalent to at least infinitely many other non-isomorphic rooted trees.
 
By virtue of the previous paragraph, for a small tree $T$ with $|[T]| > 1$ we let $C=\{(T_i, r_i)\mid i \in \omega\}$ be an infinite collection of mutually non-isomorphic rooted trees topologically equivalent to $(T,r)$. We show that $\{T_i \mid i\in \omega\}$ contains an infinite collection of non-isomorphic trees topologically equivalent to $T$. In order to achieve this goal, we show that given any triplet $(T_j, r_j), (T_k, r_k)$ and $(T_l, r_l)$ from $C$ at least one pair from the set $\{T_j, T_k, T_l\}$ are non-isomorphic. For the sake of contradiction assume otherwise and let $h_j:V(T_l) \to V(T_j)$ and $h_k:V(T_l) \to V(T_k)$ denote the respective isomorphisms - observe that $h_j(r_l)\not = r_j$ and $h_k(r_l)\not = r_k$. Since $(T,r)$, $(T_j, r_j), (T_k, r_k)$ and $(T_l, r_l)$ are of the same topological type there exist rooted minor embeddings $g_j: V(T_j) \to V(T)$, $g_k: V(T_k) \to V(T)$ and $g_l: V(T) \to V(T_l)$ witnessing this. Recall that rooted minor embeddings between small trees must map roots to roots. It follows that $g_k\circ h_k\circ g_l(r) \not = r \not =g_j\circ h_j\circ g_l(r)$. In turn, the edge $\{r,s\}$ is a fixed edge of $T$ and $g_k\circ h_k\circ g_l(r) = g_j\circ h_j\circ g_l(r) = s$ and $g_k\circ h_k\circ g_l(s) = g_j\circ h_j\circ g_l(s) = r$. The contradiction is then the following one: since $ h_j^{-1}(r_j) =g_l(s) = h_k^{-1}(r_k)$ then $h_j\circ h_k^{-1}$ establishes an isomorphism between $(T_j, r_j)$ and $(T_k, r_k)$ - impossible by our assumption.
\end{proof}

\section{Acknowledgments}

We would like to extend our sincere gratitude to the reviewers for their many helpful suggestions and comments.

\end{document}